\documentclass[12pt,reqno]{article}

\usepackage{amsmath}
\usepackage{amssymb}
\usepackage{hyperref}
\usepackage{amsthm}
\usepackage{color}
\usepackage{tikz}

\definecolor{offgreen}{rgb}{0,.60,0}

\theoremstyle{plain}
\newtheorem{theorem}{Theorem}
\newtheorem{corollary}[theorem]{Corollary}
\newtheorem{lemma}[theorem]{Lemma}

\theoremstyle{definition}

\newtheorem{example}[theorem]{Example}

\theoremstyle{remark}

\newcommand{\seqnum}[1]{\href{https://oeis.org/#1}{\underline{#1}}}

\title{Base 3/2 and Greedily Partitioned Sequences}
\author{Tanya Khovanova\\MIT \and Kevin Wu\\Conestoga High School}
\date{}

\begin{document}
\maketitle

\begin{abstract}
We delve into the connection between base $\frac{3}{2}$ and the greedy partition of non-negative integers into 3-free sequences. Specifically, we find a fractal structure on strings written with digits 0, 1, and 2. We use this structure to prove that the even non-negative integers written in base $\frac{3}{2}$ and then interpreted in base 3 form the Stanley cross-sequence, where the Stanley cross-sequence comprises the first terms of the infinitely many sequences that are formed by the greedy partition of non-negative integers into 3-free sequences.
\end{abstract}

\section{Introduction}

Historically the study of fractional bases was begun by R\'{e}nyi \cite{R} in 1957. R\'{e}nyi only used digits less than $\beta$ for base $\beta > 1$, and as a result, the integers need digits after the radix point. 

Another approach was introduced by Akiyama, Frougny, and Sakarovich \cite{AFS} and by Frougny and Klouda \cite{FK}. Here, integers less than $p$ are the digits of numbers in a rational base $\frac{p}{q} > 1$. The advantage of this approach is that integers can be represented as finite strings without using the radix point. 

This base is often called a $q$-$p$ machine and was studied as chip-firing. It was introduced by Propp \cite{JP} and widely popularized by Tanton \cite{JT} as Exploding Dots.

In this paper, we are interested in base $\frac{3}{2}$, using the digits 0, 1, and 2 to represent integers as finite strings.

This paper's other area of interest is the greedy partitioning of non-negative integers into sequences that do not contain arithmetic progressions of length 3. Such sequences are called 3-free sequences.

Odlyzko and Stanley \cite{OS} introduced the lexicographically first 3-free sequence for non-negative integers, called the Stanley sequence. This sequence $S_0$ begins as 0, 1, 3, 4, 9, 10, and so on, and can be described as integers not containing the digit 2 when written in ternary.

Gerver, Propp, and Simpson \cite{GPS} studied a greedy partition of all non-negative integers into 3-free sequences. The first sequence is the Stanley sequence $S_0$. Each next  sequence $S_i$ is constructed to be the lexicographically first 3-free sequence not containing any integers in previous sequences $S_j$, for $ j < i$. There are infinitely many such sequences. 

The first terms of these sequences form a sequence also studied by Borodin et al.~\cite{PSS}, which they called the Stanley cross-sequence. Borodin et al.~\cite{PSS} discovered a connection between base $\frac{3}{2}$ and the greedy partition of integers into sequences. Namely, they conjectured that the even non-negative integers written in base $\frac{3}{2}$ form the Stanley cross-sequence when interpreted in base 3.

Borodin et al.~\cite{PSS} made another conjecture in order to prove the conjecture above. They conjectured that the set of strings in $S_i$, when written in base 3, is the same as the set of strings achieved by adding $2i$ to strings containing 0 and 1 in base $\frac{3}{2}$.

Our main goal is to prove these conjectures, which are are now Theorems~\ref{thm:Scs} and~\ref{thm:main} respectively.

To do this, we arrange strings containing 0, 1, and 2 in a grid, so that the $0$-th row consists of strings containing only digits 0 and 1. The $i$-th row is formed by adding $2i$ to the $0$-th row in base $\frac{3}{2}$. This grid can be supplied with a fractal structure that shows the order of integers in the grid when evaluated in base 3. Using this fractal structure, we prove that the $k$-th row of the grid evaluated in base 3 contains the same numbers as the greedily constructed sequence $S_k$ above. 

Here is a roadmap for the next sections of this paper.

In Section~\ref{sec:fracbases}, we describe fractional bases. Particularly, we discuss the carrying mechanism of adding two to an integer in base $\frac{3}{2}$. In Section~\ref{sec:greedypart}, we describe the greedy partition of non-negative integers into 3-free sequences.
 
In Section~\ref{sec:grid}, we introduce an infinite grid created through addition in base $\frac{3}{2}$ and discuss properties of the infinite grid. 
In Section~\ref{sec:fractalstructure}, we explain the fractal structure that emerges when strings in the grid are interpreted in base 3.
In Section~\ref{sec:proof}, we prove Theorem~\ref{thm:interm} that is the main intermediate step to proving Theorems~\ref{thm:Scs} and~\ref{thm:main} by providing an alternative condition to greedily partitioning.

In Section~\ref{sec:twoproofs} we prove Theorems~\ref{thm:Scs} and~\ref{thm:main}, thus proving the conjectures posited by Borodin et al.~\cite{PSS}.

\section{Fractional bases}\label{sec:fracbases}

Exploding Dots is the framework in which we define the representation of integers in rational bases. 

Consider a rational number $\frac{p}{q}$, where $p > q$. We begin with boxes, placed left to right. We start with an integer $N$ and place $N$ dots in the rightmost box. Whenever a box has $p$ or more dots, those dots are exploded, removed from the box, and replaced with $q$ dots in the box directly to the left. We proceed with these explosions until none can be performed, and each box has less than $p$ dots. Then we write down the number of dots in each box from left to right. The resulting string represents $N$ in base $\frac{p}{q}$.

For example, if $p=10$ and $q=1$, we get the decimal representation of integer $N$, where each box represents a digit place.

We denote the representation of $N$ in base $\frac{p}{q}$ as $(N)_{\frac{p}{q}}$  and the evaluation of string $w$ written in base $\frac{p}{q}$ as $[w]_{\frac{p}{q}}$. For example, $(5)_{\frac{3}{2}} = 22$ and $(5)_3 = 12$, conversely $[22]_{\frac{3}{2}} = 5$ and $[12]_3 = 5$.

Exploding Dots was popularized by James Tanton \cite{JT}. But exploding dots began as a chip-firing procedure suggested by James Propp \cite{JP}.

In the chip-firing world, the procedure described above is called a $q$-$p$ machine.
In this paper, we consider a 2-3 machine, which we often refer to as base $\frac{3}{2}$. 

More formally, we can describe the representation of $N$ in base $\frac{3}{2}$ recursively. If $r$ is the remainder of $N$ modulo 3, then to get $(N)_{\frac{3}{2}}$ we concatenate  $\left(\frac{2(N-r)}{3}\right)_{\frac{3}{2}}$ with $r$.

The first few non-negative integers written in base $\frac{3}{2}$ form sequence A024629:
\[0,\ 1,\ 2,\ 20,\ 21,\ 22,\ 210,\ 211,\ 212,\ 2100,\ 2101,\ 2102,\ 2120,\  \ldots.\]
	
Given the representation of an integer $a_na_{n-1}\ldots a_1a_0$ in base $\frac{3}{2}$, we can recover the integer as 
\[[a_na_{n-1}\ldots a_1a_0]_{\frac{3}{2}} = \sum\limits_{i=0}^n a_i\left(\frac{3}{2}\right)^i.\]

Later in the paper, we have to add 2 to an integer in base $\frac{3}{2}$ many times. Now we describe how to add 2 to an integer.

Consider the rightmost zero in $X$. If there are no zeros, we prepend $X$ with a zero. To add 2, all the digits to the right of the rightmost zero, and including this zero, are reduced by 1 modulo 3, while all the digits before this zero remain unaffected. This is true because each digit will ``carry'' and in effect add a 2 to the next digit up, until the rightmost 0.

For example, $[212021]_{\frac{3}{2}} + 2 = [212210]_{\frac{3}{2}}$.

\section{The greedy partition of integers into 3-free sequences}\label{sec:greedypart}

A sequence is called \textit{3-free} if it does not contain an arithmetic progression of length 3. The \textit{Stanley sequence} $S_0$ is the lexicographically earliest 3-free sequence on the set of non-negative integers.

Sequence $S_0$ begins with 0 and 1, then skips 2 as 0, 1, and 2 form an arithmetic progression. We then add 3 and so on. The sequence is as follows:
\[0,\ 1,\ 3,\ 4,\ 9,\ 10,\ 12,\ 13,\ 27,\ 28,\ 30,\ 31,\ 36,\ 37,\ 39,\ 40,\ \ldots \]
This is sequence \seqnum{A005836} in the OEIS database \cite{OEIS}.

It is widely known that sequence \seqnum{A005836} is the sequence of integers represented without a 2 in base 3. 

Now we describe a greedy partition of non-negative integers into 3-free sequences studied in Gerver at al. \cite{GPS}. 

We start with the Stanley sequence denoted as $S_0$. The next sequence $S_1$ is constructed as the lexicographically earliest 3-free sequence on the non-negative integers not used in $S_0$. More generally, $S_n$ is constructed to be the lexicographically earliest 3-free sequence on the non-negative integers not used in any sequence $S_i$, where $i<n$.

For example, $S_1$ has to start with 2 and 5, as these are the smallest numbers not present in $S_0$. Next, we can add 6 to it, but we cannot add 7, as 5, 6, and 7 form a 3-term arithmetic progression. Neither can we add 8, as 2, 5, and 8 form an arithmetic progression. The next number must be 11. This is sequence \seqnum{A323398} in the OEIS \cite{OEIS}:
\[2,\ 5,\ 6,\ 11,\ 14,\ 15,\ 18,\ 29,\ 32,\ 33,\ 38,\ 41,\ 42,\ 45,\ 54,\ 83,\ \ldots.\]

From Borodin et al.~\cite{PSS}, \seqnum{A323398} written in base 3 are integers that contain exactly one 2 that might be followed by zeros.

It is known that the density of a 3-free sequence is 0, see \cite{Roth}. That means that the procedure described above is infinite, that is, sequences $S_n$ exist for any $n$. 

Consider the  sequence of integers that is formed by the first terms of $S_i$:
\[0,\ 2,\ 7,\ 21,\ 23,\ 64,\ 69,\ 71,\ 193,\ 207,\ \ldots\]

This sequence was named in \cite{PSS} as the \textit{Stanley cross-sequence}, and is sequence \seqnum{A265316} in the OEIS database \cite{OEIS}.

The following theorem was stated as a conjecture in \cite{PSS}.

\begin{theorem}\label{thm:Scs}
The Stanley cross-sequence written in base 3 is the same as the sequence of even non-negative integers written in base $\frac{3}{2}$.
\end{theorem}

Proving this theorem is one of the main goals of this paper. This is done in Section~\ref{sec:twoproofs}.

\section{Infinite grid}\label{sec:grid}

We introduce an infinite grid of strings, consisting of the digits 0, 1, and 2. 
The grid is created as follows. We place strings consisting only of digits 0 and 1 in the 0-th row in the natural order: that is in increasing order by value were these strings treated as integers in any integer base. Note however, that in base $\frac{3}{2}$, these strings will not all be integers or in increasing order. Every string in subsequent rows is generated by adding 2 in base $\frac{3}{2}$ to the string expressed in base $\frac{3}{2}$ directly above it.

Here we list some properties of the strings in the grid which are
proven in \cite{PSS}:

\begin{itemize}
    \item No string appears more than once in the grid.
    \item Every finite string not starting with zero, and containing digits 0, 1, and 2 appears in the grid.
    \item Each row evaluated in base $\frac{3}{2}$ is a 3-free set of numbers.
\end{itemize}

In \cite{PSS}, it was conjectured that each row $i$ interpreted in base 3 consists of numbers from greedily constructed sequence $S_{i}$ in a different order. 

Here is an upper left corner of grid $G$.

\[
\begin{matrix}
    0  & 1 & 10 & 11 & 100 & 101 & \dots \\
    2 & 20 & 12 & 200 & 102 & 120 & \dots \\
    21 & 22 & 201 & 202 & 121 & 122 & \dots \\
    210 & 211 & 220 & 221 & 2010 & 2011 & \dots \\
    \vdots
\end{matrix}
\]

This grid $G$ is the main object of our study.

Consider coordinates $(i,j)$ in the grid $G$, where $i$ is the row number and $j$ is the column number, starting from zero. We denote an entry with coordinate $(i,j)$ as $G(i,j)$. By our construction 
\[[G(i+1,j)]_{\frac{3}{2}} = [G(i,j)]_{\frac{3}{2}} + 2.\]

The concatenation of two strings $Y$ and $X$ is denoted by the following overline notation: $\overline{YX}$.

\begin{lemma}[Prefix Property]
Any string in row $i$ prepended with any string containing zeros and ones exists as another string in the same row.
\label{lemma:prefix}
\end{lemma}
\begin{proof}
Take string $G(i,j)$ of length $\ell$, which we prepend with string $Y$ consisting of zeros and ones to get a string $\overline{YG(i,j)}$. We prove that $\overline{YG(i,j)}$ is in row $i$.

Prepend $G(0,j)$ with zeros such that $G(0,j)$ has length $\ell$, and call this $G'(0,j)$. As $\overline{YG'(0,j)}$ consists of zeros and ones, $\overline{YG'(0,j)}$ is in row $0$. By definition 
\[[G(i,j)]_{\frac{3}{2}} = [G(0,j)]_{\frac{3}{2}}+2i = [G'(0,j)]_{\frac{3}{2}}+2i.\]
We may prepend both $G(i,j)$ and $G'(0,j)$ with $Y$ in the previous equality, as both are length $\ell$. Thus, we have $[\overline{YG(i,j)}]_{\frac{3}{2}} = [\overline{YG'(0,j)}]_{\frac{3}{2}} + 2i$, so $\overline{YG(i,j)}$ is in row $i$.
\end{proof}

For example, if a row contains 200, that row will also contain strings 1200, 10200, 11200, etc.

The \textit{main suffix} of the string $G(i,j)$ is $G(i,j)$ with the initial string of zeros and ones chopped off. The main suffix is either empty or begins with $2$.

\begin{lemma}[Suffix Property]
Two strings with the same main suffix belong to the same row. 
\end{lemma}
\begin{proof}
Consider two strings $\overline{YX}$ and $\overline{WX}$ with the same main suffix $X$. As any string appears in the grid, the string $X$ has to be in some row $i$. From the prefix property, $\overline{YX}$ and $\overline{WX}$ are both in row $i$.
\end{proof}

\section{Fractal structure}\label{sec:fractalstructure}

Now we take grid $G$ and add some directional lines between the entries in a recursive manner, creating a fractal structure. Grid $G$ is the zeroth step, so sometimes we call it  $G_0$.

\subsection{The first iteration}

In the first iteration, we connect some strings of digits in this grid using black color, $c_0$.

We connect two strings $x$ and $y$ with a directed segment of color $c_0$ if they satisfy the following condition: Their digits are the same except for the last digit and $y = x+1$. We show the top left portion of the result in Figure~\ref{fig:g1} in black color. We omit the arrows for directed segments to not clutter the picture. We denote this grid together with the directed segments as $G_1$.

\begin{figure}[ht]
    \centering
\begin{tikzpicture} [x = .5cm, y = .5cm]
    \foreach \x in {0,1,...,15} {
        \foreach \y in {0,1,...,17} {
            \fill[color=black] (\x,\y) circle (0.1);
            
        }
    }
    \foreach \x in {0,1,...,7} {
        \foreach \y in {0,1,...,5} {
            \draw (2*\x,3*\y+2) -- (2*\x+1,3*\y+2);
            \draw (2*\x,3*\y) -- (2*\x+1,3*\y);
            \draw (2*\x,3*\y) -- (2*\x+1,3*\y+1);
            \draw (2*\x,3*\y+1) -- (2*\x+1,3*\y+2);
        }
    }
\end{tikzpicture}
    \caption{$G_1$.}
    \label{fig:g1}
\end{figure}

As we see in the picture, the black connected figures form two shapes that we call \textit{upperZ} and \textit{lowerZ}, because an upperZ and a lowerZ right below it together look like a Z. When it is unspecified whether a connected figure is an upperZ or a lowerZ, we refer to it simply as a \textit{halfZ}. 

The initial sequence of digits that two strings share is called their \textit{longest common prefix}.
There are 3 strings making up each halfZ. As we mentioned before, all the strings in the same halfZ have the same digits except the last digit. If three strings are in the same halfZ, their longest common prefix is simply one of them with the last digit removed. We refer to this longest common prefix as \textit{lcpHZ}.

In the upperZ, we have a segment starting from the upper left corner to the upper right corner, and then to the bottom left corner, forming an acute angle at the top right corner of the shape. The last digits of the strings progress from 0 to 1 to 2, respectively. 
Similarly, in the lowerZ we have a segment starting from the upper right corner to the bottom left corner, and then to the bottom right corner, forming an acute angle in the bottom left corner of the shape. The last digits of the strings progress from 0 to 1 to 2, respectively as well.

The following lemma describes the positions of upperZs and lowerZs.

\begin{lemma}
The upperZ corresponds to a path from $(3a,2b)$ to $(3a,2b+1)$ to $(3a+1,2b+1)$. The lowerZ corresponds to a path from $(3a+1,2b+1)$ to $(3a+2,2b)$ to $(3a+2,2b+1)$. Every entry in the grid belongs to exactly one halfZ.
\end{lemma}

\begin{proof}
We know that $[G(i+1,j)]_\frac{3}{2} = [G(i,j)]_\frac{3}{2} + 2$.

The last digit of the 0-th row alternates between 0 and 1, therefore, the last digit of $G(0,2b)$ is 0 and the last digit of $G(0,2b+1)$ is 1 and both strings $G(0,2b)$ and $G(0,2b+1)$ have the same digits except for the last one. There is a directed segment from $(0,2b)$ to $(0,2b+1)$, as they are part of the same halfZ. 

As $G(1,2b) = G(0,2b) + 2$ and $G(0,2b)$ ends in 0, the only digit that will change is the last digit, from 0 to 2. Thus there will be a directed segment from $(0,2b+1)$ to $(1,2b)$, completing the halfZ. The lemma is true for the upperZ corresponding to $a=0$.

Next, consider $G(1,2b+1)$ and $G(2,2b)$. As $[G(1,2b+1)]_{\frac{3}{2}} = [G(0,2b+1)]_{\frac{3}{2}} + 2 = [G(0,2b)]_{\frac{3}{2}} +3$, we know $G(1,2b+1)$ ends in 0. Also,  $[G(2,2b)]_{\frac{3}{2}} = [G(1,2b)]_{\frac{3}{2}} + 2 = [G(0,2b)]_{\frac{3}{2}} + 4$. In addition, $[G(2,2b+1)]_{\frac{3}{2}} = [G(0,2b+1)]_{\frac{3}{2}} + 4 = [G(0,2b)]_{\frac{3}{2}} + 5$.

This means $[G(2,2b)]_{\frac{3}{2}} = [G(1,2b+1)]_{\frac{3}{2}}+1$, and $[G(2,2b+1)]_{\frac{3}{2}} = [G(1,2b+1)]_{\frac{3}{2}}+2$. As $G(1,2b+1)$ ends in 0, the directed segments must go from $(1,2b+1)$ to $(2,2b)$ to $(2,2b+1)$, forming a lowerZ. The lemma is true for lowerZ corresponding to $a=0$.

Finally, observe that $[G(i+3,j)]_{\frac{3}{2}} = [G(i,j)]_{\frac{3}{2}}+6$. That means $G(i+3,j)$ and $([G(i,j)]_{\frac{3}{2}}+6)_{\frac{3}{2}}$ have the same last digit. If $G(i,j)$ and $G(n,m)$ share all the digits but the last one, then $G(i+3,j)$ and $G(n+3,m)$ share all the digits but the last one.Thus, the shapes for halfZs are vertically periodic with period 3. As the first three rows are partitioned into halfZs, by periodicity the whole grid is partitioned as well.
\end{proof}

The grid together with black halfZs is $G_1$. 

\subsection{Other iterations}

For the next step of our iteration, we choose another color $c_1$, which is red in Figure~\ref{fig:g4}. Now we connect the halfZs using the same procedure as before on the halfZs, where each halfZ is represented by its respective lcpHZ. This time we connect halfZs such that the corresponding entries differ in the last digit of lcpHZs only. The new grid with new connections is denoted $G_2$. 

We continue this procedure with color $c_m$, which connects halfZs of color $c_{m-1}$ using the same procedure as before. In each iteration of the process, the three strings corresponding to the same halfZ are replaced with one string, namely lcpHZ. The new strings are connected by the same procedure depending on the last digit. 
Note that the lcpHZ for a halfZ of color $c_m$ will have $m$ fewer digits than the lcpHZ for a halfZ of color $c_0$.

The resulting picture is denoted $G_{m+1}$. In Figure~\ref{fig:g4} we show the top left corner of $G_4$, where we use black as $c_0$, red as $c_1$, green as $c_2$, and orange as $c_3$.

\begin{figure}[h]
    \centering
\begin{tikzpicture} [x = .5cm, y = .5cm]
    \foreach \x in {0,1,...,15} 
    {
        \foreach \y in {0,1,...,17} 
        {
            \fill[color=black] (\x,\y) circle (0.1);
            
        }
    }
    \foreach \x in {0,1,...,7} 
    {
        \foreach \y in {0,1,...,5} 
        {
            \draw (2*\x,3*\y+2) -- (2*\x+1,3*\y+2);
            \draw (2*\x,3*\y) -- (2*\x+1,3*\y);
            \draw (2*\x,3*\y) -- (2*\x+1,3*\y+1);
            \draw (2*\x,3*\y+1) -- (2*\x+1,3*\y+2);
            \draw (2*\x,3*\y) -- (2*\x+1,3*\y+1);
        }
    }
    \foreach \x in {0,1,...,3}
    {
        \draw [line width = .03 cm] [red](4*\x+.5,16.5) -- (4*\x+2.5,16.5);
        \draw [line width = .03 cm][red](4*\x+.5,15.5) -- (4*\x+2.5,16.5);
        \draw [line width = .03 cm][red](4*\x+.5,13.5) -- (4*\x+2.5,15.5);
        \draw [line width = .03 cm][red] (4*\x + .5,13.5) -- (4*\x+2.5,13.5);
        \draw [line width = .03 cm][red](4*\x+.5,12.5) -- (4*\x+2.5,12.5);
        \draw [line width = .03 cm][red](4*\x+.5,10.5) -- (4*\x+2.5,12.5);
        \draw [line width = .03 cm][red](4*\x+.5,9.5) -- (4*\x+2.5,10.5);
        \draw [line width = .03 cm][red] (4*\x + .5,9.5) -- (4*\x+2.5,9.5);
        \draw [line width = .03 cm][red](4*\x+.5,7.5) -- (4*\x+2.5,7.5);
        \draw [line width = .03 cm][red](4*\x+.5,6.5) -- (4*\x+2.5,7.5);
        \draw [line width = .03 cm][red] (4*\x + .5,4.5) -- (4*\x+2.5,4.5);
        \draw [line width = .03 cm][red](4*\x+.5,4.5) -- (4*\x+2.5,6.5);
        \draw [line width = .03 cm][red](4*\x+.5,3.5) -- (4*\x+2.5,3.5);
        \draw [line width = .03 cm][red](4*\x+.5,1.5) -- (4*\x+2.5,3.5);
        \draw [line width = .03 cm][red](4*\x+.5,.5) -- (4*\x+2.5,1.5);
        \draw [line width = .03 cm][red] (4*\x + .5,.5) -- (4*\x+2.5,.5);
    }
    \foreach \x in {0,1}
    {
        \draw [line width = .05 cm][offgreen] (8*\x+1.5,2.5)--(8*\x+5.5,2.5);
        \draw [line width = .05 cm][offgreen] (8*\x+1.5,1)--(8*\x+5.5,2.5);
        \draw [line width = .05 cm][offgreen] (8*\x+1.5,5.5)--(8*\x+5.5,5.5);
        \draw [line width = .05 cm][offgreen] (8*\x+1.5,5.5)--(8*\x+5.5,7);
        \draw [line width = .05 cm][offgreen] (8*\x+1.5,7)--(8*\x+5.5,10);
        \draw [line width = .05 cm][offgreen] (8*\x+1.5,10)--(8*\x+5.5,10);
        \draw [line width = .05 cm][offgreen] (8*\x+1.5,11.5)--(8*\x+5.5,11.5);
        \draw [line width = .05 cm][offgreen] (8*\x+1.5,11.5)--(8*\x+5.5,14.5);
        \draw [line width = .05 cm][offgreen] (8*\x+1.5,14.5)--(8*\x+5.5,16);
        \draw [line width = .05 cm][offgreen] (8*\x+1.5,16)--(8*\x+5.5,16);
    }
    \draw [line width = .06 cm][orange] (3.5,1.75)--(11.5,6.25);
    \draw [line width = .06 cm][orange] (3.5,6.25)--(11.5,6.25);
    \draw [line width = .06 cm][orange] (3.5,8.5)--(11.5,8.5);
    \draw [line width = .06 cm][orange] (3.5,8.5)--(11.5,13);
    \draw [line width = .06 cm][orange] (3.5,13)--(11.5,15.25);
    \draw [line width = .06 cm][orange] (3.5,15.25)--(11.5,15.25);
\end{tikzpicture}
    \caption{$G_4$.}
    \label{fig:g4}
\end{figure}

As one can see in the figure, the red segments form halfZs in the same way as the black segments, just bigger. This similarity is due to the fractal structure of our grid. We want to explicitly prove this, so we define the zooming-out procedure.

In the \textit{zooming-out} procedure, we create a new grid. We replace each halfZ with its lcpHZ. One new row in the new grid corresponds to one row of either upperZs or lowerZs. The numbers corresponding to upperZs are placed in the even rows in the same order from left to right, and the numbers corresponding to lowerZs are placed in the odd rows from left to right.

\begin{lemma}
When zooming-out from $G_0$, we get $G_0$.
\end{lemma}

\begin{proof}
Let the \textit{zooming-out} procedure on $G $ create $G'$. The first row of $G$ consists of terms with only 0 and 1 in increasing order evaluated in base 2. The first row of $G'$ is the sequence of lcpHZs for the first row of upperZs in $G_0$, in increasing order base 2. Thus, the first row of $G'$ is identical to the first row of $G$.

Consider an upperZ from $(3a,2b)$ to $(3a,2b+1)$ to $(3a+1,2b)$, where $G(3a,2b)$ ends in zero. The string $G(3a+2,2b)$ which in base $\frac{3}{2}$ equals $[G(3a,2b)]_\frac{3}{2}+ 4$ belongs to the lowerZ below it. As $(4)_\frac{3}{2} = 21$, the value of lcpHZ for the upperZ in base $\frac{3}{2}$ is increased by 2 to become the value of lcpHZ for the lowerZ in base $\frac{3}{2}$. The string $G(3(a+1),2b)$ which equals $[G(3a,2b)]_\frac{3}{2}+ 6$ when evaluated in base $\frac{3}{2}$, belongs to the next upperZ below. As $(6)_\frac{3}{2} = 210$, the value of the lcpHZ for this upperZ is $[21]_{\frac{3}{2}} = 4$ greater than the value of the lcpHZ for the upperZ two rows above. The value of the lcpHZs in base $\frac{3}{2}$ always increases by $2$ with respect to the lcpHZ of the halfZ above. Thus, the lcpHZs follow the original construction of $G_0$, and form $G_0$ themselves.
\end{proof}

\begin{corollary}\label{cor:zoomvalue}
The lcpHZ for the upperZ from $(3a,2b)$ to $(3a,2b+1)$ to $(3a+1,2b)$ is $G(2a,b)$ and the lcpHZ for the lowerZ from $(3a+1,2b+1)$ to $(3a+2,2b)$ to $(3a+2,2b+1)$ is $G(2a+1,b)$.
\end{corollary}

A halfZ in $G_n$ with $x$ halfZs to its left and $y$ halfZs above it, has its lcpHZ equal to $G_{n+1}(x,y)$.

\subsection{Order of numbers interpreted in base 3}

The zeroth row evaluated in base $\frac{3}{2}$ is not in increasing order. The order was studied in \cite{PSS}. But it is true that every row is in the same order as the zeroth row as we add 2 to a row in base $\frac{3}{2}$ to get to the next row.

Note that the zeroth row evaluated in base $3$ is in increasing order. However, the other rows are not in increasing order when evaluated in base 3. As an example, in row $1$, we see 20 appearing to the left of 12, which are 6 and 5 respectively, when evaluated in base 3.

It is of great interest to understand how the strings in grid $G$ are ordered when evaluated in base 3.

\begin{theorem}\label{thm:order}
By traversing our directed lines, each time prioritizing the lowest color, we follow the numbers in increasing order base 3.
\end{theorem}
\begin{proof}
Because of the zooming-out action, we ensure the first digits are always the smallest possible, and the last digits go up in order from our directed line segments. Thus, the numbers will be in numerical order.
\end{proof}

As an example of this action, we describe the beginning of the ordering, where we use Figure~\ref{fig:g4} as visual help.
We begin by following the black upperZ, which contains 0, 1, 2 in this order. The next lowest color, red, instructs us to move to the black upperZ right of the first black upperZ, which contains strings 10, 11, 12 in this order. The last part of the red upperZ instructs us to move to the lowerZ below the first black upperZ, which contains strings 20, 21, 22 in this order. As the red upperZ has been traversed, the green segments show us which red halfZ to be traversed next. The starting point of the next red halfZ is a black halfZ, which we traverse the same way as before. We continue similarly. In each step, the strings are in increasing order if interpreted in base 3.

We built each row in the grid using base $\frac{3}{2}$. Theorem~\ref{thm:order} gives us the ordering of strings in the grid if they are interpreted in base 3. Thus the grid somehow connects these two bases. This connection is discussed in the next section.

Meanwhile, we need the following lemma that will be useful later.

\begin{lemma}\label{lemma:minus1}
For every $G(m,n)$ in the grid, $([G(m,n)]_3-1)_3$  is from row $r \geq m-1$.
\end{lemma}
\begin{proof}
In all future cases, $G(p,q)$ refers to the lcpHZ of the halfZ to which $G(m,n)$ belongs.

If $G(m,n)$ ends in $1$ or $2$, $([G(m,n)]_3-1)_3$ belongs to the same halfZ, so $r=m$ or $r = m-1$, thus $r \geq m-1$.

If $G(m,n)$ ends in exactly one zero, then $([G(p,q)]_3-1)_3$ is the lcpHZ of the halfZ to which $([G(m,n)]-1)_3$ belongs. Because $G(m,n)$ ends in exactly one zero, we have $G(p,q)$ does not end in $0$, so $G(p,q)$ and $([G(p,q)]_3-1)_3$ belong to the same halfZ, and $([G(p,q)]_3-1)_3$ belongs to row $p-1$ or $p$. Because $([G(m,n)]_3-1)_3$ ends in 2, zooming out from $([G(p,q)]_3-1)$ and $G(p,q)$ gives us that $([G(m,n)]_3-1)_3$ is from some row $r \geq m-1$.

Now we assume that for any $G(m,n)$ that ends in $k$ zeros, it is true that $([G(m,n)]_3-1)_3$ is from row $r \geq m-1$. We use induction to prove that for any $G(m,n)$ that ends in $k+1$ zeros, it is true that $([G(m,n)]_3-1)_3$ is from row $r \geq m-1$. 

Assume $G(p,q)$ ends in $k$ zeros, then $([G(p,q)]_3-1)_3$ is from row $\geq p-1$ by our induction hypothesis.

We know $[\overline{G(p,q)0}]_3 = [G(p,q)]_3$ and $\overline{([G(p,q)]_3-1)_{3} 2}$= $([G(m,n)]_3-1)_3$. As $([G(p,q)]_3-1)_3$ is from row $\geq p-1$, the zooming-out action gives us that $([G(m,n)]-1)_3$ is from some row $r \geq m-1$. In other words, for $G(m,n)$ ending in $k+1$ zeros, $([G(m,n)]_3-1)_3$ is from some row $r \geq m-1$.
\end{proof}

We also need the following lemma.

\begin{lemma}\label{lemma:zerocolumn}
When a row is evaluated in ternary, the smallest number in that row corresponds to the string in the zeroth column. Additionally, each $[G(j,0)]_3<[G(i,0)]_3$ for $j<i$.
\end{lemma}

\begin{proof}
As the numbers in the grid are non-negative, the lemma holds for the $0$-th row, as $G(0,0) = 0$.
The lemma also holds for row $1$, as $[G(1,0)]_3 = 2$ and $0$ and $1$ both appear in the $0$-th row. We also have that $[G(0,0)]_3 = 0 < [G(1,0)]_3 = 2$.

Induction hypothesis: Assume every row  $j<i$ evaluated in base 3, has its smallest number in the zeroth column, and that each $[G(j,0)]_3<[G(j+1,0)]_3$. We wish to show that row $i$ evaluated in base 3, must also have its smallest number in the zeroth column and that $[G(i,0)]_3<[G(i+1,0)]_3$.

If $i = 3a$ or $3a+2$, then the strings in row $i$ with the last digit truncated will all be from the same row, either $2a$ or $2a+1$ respectively. As the smallest number in row $2a$ is $[G(2a,0)]_3$ and the smallest number in row $2a+1$ is $[G(2a+1,0)]_3$, it follows that the smallest number in row $i$ is $[G(i,0)]_3$.

If $i = 3a+1$, the strings in row $i$ with the last digit truncated are either from row $2a$ or $2a+1$. Strings from the even columns of row $i$ correspond with elements from $2a$ and from odd columns with $2a+1$. Thus the smallest number of $\{[G(i,0)]_3,[G(i,2)]_3,[G(i,4)]_3,...\}$ is $[G(i,0)]_3$ and the smallest number of $\{[G(i,1)]_3,[G(i,3)]_3,[G(i,5)]_3,...\}$ is $[G(i,1)]_3$. We have \[[G(i,0)]_3 = 3[G(2a,0)]_3+2 < 3[G(2a+1,0)]_3 = [G(i+1,0)]_3\] as $[G(2a,0)]_3 < [G(2a+1,0)]_3$ by our induction hypothesis, so the smallest number in row $i$ is $[G(i,0)]_3$.

Now we show that $[G(i,0)]_3<[G(i+1,0)]_3$. If they belong to the same upperZ, then $[G(i,0)]_3 = 3[G(2a,0)]_3 < 3[G(2a,0)]_3+2 = [G(i+1,0)]_3$, and the inequality holds. If they do not belong to the same upperZ, then $[G(i,0)]_3 = 3[G(j,0)]_3+r_1 < 3[G(j+1,0)]_3+r_2 = [G(i+1,0)]_3$, where $r_1$ and $r_2 \in {0,1,2}$.

Thus, $[G(i,0)]_3$ is the smallest number in row $i$ and $[G(i,0)]_3<[G(i+1,0)]_3$ for every $i$.
\end{proof}

\section{Theorem~\ref{thm:interm} and its proof}\label{sec:proof}

Now we examine specific properties of the rows of $G$. The paper \cite{PSS}, that our paper was chiefly inspired by, showed that each row is 3-free. Recall that a set is 3-\textit{free} if it does not contain an arithmetic progression of length 3. 

One of the goals of this paper is to prove the following theorem that was conjectured in \cite{PSS}.

\begin{theorem}\label{thm:main}
The infinite set of numbers, which are the strings in row $i$ interpreted in base 3, consists of the same numbers as sequence $S_i$, the $i$-th sequence of the greedy partition of non-negative integers into 3-free sequences.
\end{theorem}

We prove this theorem later in Section~\ref{sec:twoproofs}. Instead, we prove a related theorem.

\begin{theorem}\label{thm:interm}
Each row of the grid in base 3 has the property that every term in row $i$ of the grid can be represented as the last term of a 3-term arithmetic progression with the first two elements in row $j$ of the grid, for any $j < i$.
\end{theorem}

Suppose the set of non-negative integers is divided into disjoint sets $R_i$, where $i \geq 0$, such that every term in $R_i$ can be represented as the last term of a 3-term arithmetic progression with the first two elements in $R_j$, for any $j < i$. The theorem claims that each row $i$ contains the same terms as $R_i$.

The proof of this theorem is done in separate lemmas, each for a different row, using induction for the overall case. 

\subsection{Rows 0 and 1}

It is well known that the set of numbers represented without the digit $2$ in base $3$ form the lexicographically first 3-free sequence \cite{OS}. As row $0$ is exactly this set of numbers, row $0$ evaluated in base 3 is equal to $R_0$.

Now we study row 1. The following lemma is proven in \cite{PSS}.

\begin{lemma}
The strings in row 1 can be described as follows. Each string has exactly one digit $2$, followed by any number of the digit $0$ and preceded by an arbitrary sequence created from the digits $0$ and $1$.
\end{lemma}

We use the following standard notation to describe a string of digits  \[x = d_1^{e_1} d_2^{e_2} \cdots d_k^{e_k},\] where $d_i$ is a digit $0, 1, $ or $2$, and each exponent $e_i$ describes the length of the run for each $d_i$. For example the string 01002100 is represented as $0^1 1^1 0^2 2^1 1^1 0^2$.

\begin{lemma}
Row 1, evaluated in base 3, contains the same set of integers as $R_1$.
\end{lemma}
\begin{proof}
We need to show that $x$ from row $k>1$ forms an arithmetic sequence with two smaller elements $a$ and $b$ of row $1$ when all of them are evaluated in base 3. Without loss of generality we assume that $[a]_3<[b]_3$. We know that $x$ has at least one digit 2 followed by digits that are not all 0. 

We parse $x$ into three sections, $x_1$, $x_2$, and $x_3$, where we denote the concatenation using the overline $x = \overline{x_1 x_2 x_3}$. We pick $x_3$ consisting of the longest number of digits $0$ at the end of $x$. If there are no zeros, then $x_3$ is empty.

If the rightmost non-zero digit of $x$ is 2, then let $x_2 = 2$. Otherwise $x_2$ ends with 1. We define $x_2$ as either $2^1 0^j 1^k$ where $j \geq 0$ and $k> 0$, or $1^1 0^j 1^k$, where $j, k > 0$. The string $x_1$ is defined to be simply the remaining digits in $x$ to the left.

Now we construct $a = \overline{a_1 a_2 a_3}$ and $b = \overline{b_1 b_2 b_3}$, where the length of $a_i$ equals the length of $b_i$ equals the length of $x_i$ for $i= 1, 2, 3$, and leading zeros are permitted. We will show there are $a$ and $b$ such that $[a_i]_3 \leq [b_i]_3 \leq [x_i]_3$, where $[a_i]_3$, $[b_i]_3$, and $[x_i]_3$ for $i = 1,2,3$ form an arithmetic sequence. This would imply that $[a]_3$, $[b]_3$, and $[x]_3$ form an arithmetic progression.

We define $a_3$ = $b_3$ = $x_3$. Thus $[a_3]_3$, $[b_3]_3$, and $[x_3]_3$ form an arithmetic sequence, as $x_3$ consists only of 0s. 

Next, in digit places where $x_1$ has 0 or 1, we use the same digit as $x_1$ in this digit place for both $a_1$ and $b_1$. For the digit places of $x_1$ that are 2, these digit places are 0 in $a_1$ and 1 in $b_1$. This construction results in an arithmetic sequence $[a_1]_3$, $[b_1]_3$, and $[x_1]_3$. If $x_1$ does not contain a 2, then $a_1 = b_1 = x_1$.
    
Now it remains only to construct $a_2$ and $b_2$. Consider caseworks on the various different possibilities for $x_2$. In each case the difference is $[x_2]_3 - [b_2]_3 = [b_2]_3 - [a_2]_3 = [0^1  1^{j+k}]_3$. 

\textbf{Case 0:} $x_2=2$.

Simply let $a_2 = b_2 = 2$. Clearly $[a_2]_3$, $[b_2]_3$ and $[x_2]_3$ form an arithmetic sequence.

\textbf{Case 1:} $x_2 = 2^1 0^j 1^k$, where $j,k>0$.

Let $a_2 = 0^{j+k} 2^1$, and $b_2 = 1^j 2^1 0^k$. This gives us $[b_2]_3 - [a_2]_3 = [x_2]_3 - [b_2]_3 = [0^1 1^{j+k}]_3$, thus $[a_2]_3$, $[b_2]_3$ and $[x_2]_3$ form an arithmetic sequence.

\textbf{Case 2:} $x_2 = 2^1 1^k$ where $k>0$.

Let $a_2 = 1^k 2^1$ and $b_2 = 2^1 0^k$. This gives us $[b_2]_3 - [a_2]_3 = [x_2]_3 - [b_2]_3 = [0^1 1^{j+k}]_3$, thus $[a_2]_3$, $[b_2]_3$ and $[x_2]_3$ form an arithmetic sequence. 

\textbf{Case 3:} $x_2 = 1^1 0^j 1^k$, where $j,k>0$.

Let $a_2 = 0^{j+k} 2^1$ and $b_2 = 0^1 1^{j-1} 2^1 0^k$. This gives us $[b_2]_3 - [a_2]_3 = [x_2]_3 - [b_2]_3 = [0^1 1^{j+k}]_3$, thus $[a_2]_3$, $[b_2]_3$ and $[x_2]_3$ form an arithmetic sequence. 

Now we see that $[a]_3 < [b]_3$, except when $[a_2]_3 = [b_2]_3 = [x_2]_3 = 2$ and $[a_1]_3 = [b_1]_3 = [x_1]_3$. However, this indicates that $x_2 = 2$, so $x$ is in row 1, which is a trivial exception. Also, $a_1$ and $b_1$ do not contain a 2; $a_2$ and $b_2$ contain one 2 each followed by zeros, and $a_3$ and $b_3$ contain only zeros. Hence, $a$ and $b$ are in row 1. Thus the proof is finished.
\end{proof}
Here is an example of the procedure described above.
\begin{example}
\[x = 11102010220102110110011000.\]
Then $x_1 = 111020102201021101$, $x_2 = 10011$, $x_3 = 000$. 
\[a_3 = b_3 = x_3 = 000. \]
For each digit where $x_1$ has a 2, $a_1$ will have a 0, so \[a_1 = 111000100001001101.\]
For each digit where $x_2$ has a 2, $b_1$ will have a 1, so \[b_1 = 111010101101011101.\]
We would like the difference between $x_2$ and $b_2$ in base 3 to be $0111$, so \[b_2 = 01200 \] and \[[a_2]_3 = [b_2]_3 - [0111]_3 = [00002]_3.\]
Thus $a = 11100010000100110100002000$ and $b = 11101010110101110101200000$. Both are from row 1 and form an arithmetic progression with $x$ evaluated in base 3.
\end{example}

\subsection{Inductive hypothesis}
Given our base cases row 0 and row 1, we now assume that each row $i<j$ evaluated in base 3 is the same set of numbers as $R_i$. We wish to show row $j$ evaluated in base 3 is the same set of numbers as $R_j$, which is done in three lemmas below for $j = 3a$, $3a+1$, and $3a+2$. This can be done by showing that any element from a row greater than row $j$ is the last term of a 3-term arithmetic progression with two elements of row $j$ when all of them are evaluated in base 3.

\subsection{Row 3a}

\begin{lemma}\label{lemma:3a}

If for $i < 3a$ every row $i$ evaluated in base 3 is the same set of numbers as $R_i$, then row $3a$ evaluated in base 3 is the same set of numbers as $R_{3a}$.
\end{lemma}
\begin{proof}
Row $3a$ consists of the elements $G(3a,2b)$ and $G(3a,2b+1)$ for $b \in \mathbb{N}_0$. From Corollary~\ref{cor:zoomvalue} we have that $[G(3a,2b)]_3 = 3[G(2a,b)]_3$ and $[G(3a,2b+1)]_3 = 3[G(2a,b)]_3+1$. 

Now consider an element $G(m,n)$ where $m>3a$. 

We start with a special case $m= 3a+1$ and $n = 2b$. Then we have $[G(m,n)]_3 = 3[G(3a,b)]_3+2$, which forms an arithmetic sequence with $[G(3a,2b)]_3$ and $[G(3a,2b+1)]_3$. 

For all other $m>3a$, we have $[G(m,n)]_3 = 3[G(p,q)]_3 + r$, where $r \in {0,1,2}$ and $p>2a$. By our induction hypothesis, row $2a$ evaluated in base 3 is the same set of numbers as $R_{2a}$, so there are two elements $[G(2a,b_1)]_3$ and $[G(2a,b_2)]_3$ that form an arithmetic progression with $[G(p,q)]_3$.

If $r = 0$, we have $[G(3a,2b_1)]_3$, $[G(3a,2b_2)]_3$, and $[G(m,n)]_3$ form an arithmetic sequence.

If $r = 1$, we have $[G(3a,2b_1+1)]_3$, $[G(3a,2b_2+1)]_3$, and $[G(m,n)]_3$ form an arithmetic sequence.

If $r = 2$, we have $[G(3a,2b_1)]_3$, $G[(3a,2b_2+1)]_3$, and $[G(m,n)]_3$ form an arithmetic sequence.

Thus, two elements from row $3a$ evaluated in base 3 form an arithmetic sequence with any $[G(m,n)]_3$ from a later row, so row $3a$ evaluated in base 3 is the same set of numbers as $R_{3a}$.
\end{proof}

Intuitively, we can think of the proof as follows. Elements from row $2a$ appended with $0$ or $1$ make up the elements of row $3a$. We know we can find an arithmetic sequence with elements of row $2a$ evaluated in base 3 and an arbitrary $[G(p,q)]_3$. The last digit $r$ of $[G(m,n)]_3 = [G(p,q)]_3 + r$ determines the last digit of our two elements from row $3a$, either in progression $(0,0,0)$, $(1,1,1)$, or $(0,1,2)$.

\subsection{Row 3a+2}
The proof for the following lemma proceeds in almost exactly identical fashion to the previous proof.

\begin{lemma}\label{lemma:3a+2}
If for $i < 3a+2$ every row $i$ evaluated in base 3 is the same set of numbers as $R_i$, then row $3a+2$ evaluated in base 3 is the same set of numbers as $R_{3a+2}$.
\end{lemma}
\begin{proof}
Row $3a+2$ consists of the elements $G(3a+2,2b)$ and $G(3a+2,2b+1)$ for $b \in \mathbb{N}_0$. From Corollary~\ref{cor:zoomvalue} we have that $[G(3a+2,2b)]_3 = 3[G(2a+1,b)]_3+1$ and $[G(3a+2,2b+1)]_3 = 3[G(2a+1,b)]_3+2$. 
Now consider an element $G(m,n)$ where $m>3a+2$. We have $[G(m,n)]_3 = 3[G(p,q)]_3 + r$, where $r \in {0,1,2}$ and $p>2a$. By our induction hypothesis, row $2a+1$ evaluated in base 3 is the same set of numbers as $R_{2a+1}$, so there are two elements $[G(2a+1,b_1)]_3$ and $[G(2a+1,b_2)]_3$ that form an arithemetic sequence with $[G(p,q)]_3$. 

If $r = 0$, we have $[G(3a+2,2b_1+1)]_3$, $[G(3a+2,2b_2)]_3$, and $[G(m,n)]_3$ form an arithmetic sequence.

If $r = 1$, we have $[G(3a+2,2b_1)]_3$, $[G(3a+2,2b_2)]_3$, and $[G(m,n)]_3$ form an arithmetic sequence.

If $r = 2$, we have $[G(3a+2,2b_1+1)]_3$, $[G(3a+2,2b_2+1)]_3$, and $[G(m,n)]_3$ form an arithmetic sequence.

Thus, two elements from row $3a+2$ evaluated in base 3 form an arithmetic sequence with any $[G(m,n)]_3$ from a later row, so row $3a+2$ evaluated in base 3 is the same set of numbers as $R_{3a+2}$.
\end{proof}

Intuitively, we can think of the proof as follows. Elements from row $2a+1$ appended with $1$ or $2$ make up the elements of row $3a+2$. We know we can find an arithmetic sequence with elements of row $2a+1$ evaluated in base 3 and an arbitrary $[G(p,q)]_3$. The last digit $r$ of $[G(m,n)]_3 = [G(p,q)]_3 + r$ determines the last digit of our two elements from row $3a+2$, either in progression $(2,1,0)$, $(1,1,1)$, or $(2,2,2)$.

\subsection{Row 3a+1}

The following notation is introduced for convenience. For a given string $x_0$, we define $x_n$ to be the prefix of $x_0$, where the last $n$ digits are removed and $x_{-n}$ to be the suffix of $x_0$ consisting of the the last $n$ digits. That means, for any $n$: \[[x_0]_3 = 3^{n}[x_n]_3+[x_{-n}]_3.\]

\begin{lemma} \label{lemma:3a+1}
If for $i < 3a+1$ every row $i$ evaluated in base 3 is the same set of numbers as $R_i$, then row $3a+1$ evaluated in base 3 is the same set of numbers as $R_{3a+1}$.
\end{lemma}
\begin{proof}
We want to show that any $x_0$ from row $m>3a+1$ evaluated in base 3 is the last term of a 3-term arithmetic sequence with two terms of row $3a+1$ evaluated in base 3. Call the two terms from row $3a+1$ we are looking for $c_0<d_0$. These terms yield the equation,
\begin{equation}
[d_0]_3-[c_0]_3 = [x_0]_3-[d_0]_3.
\label{someEQ}
\end{equation}

\textbf{Case A:} If $x_{-1}$ = 0, set $c_{-1} = d_{-1} = 0$. Thus $c_0$ and $d_0$ both belong to lowerZs, so $c_1$ and $d_1$ belong to the same row, and Equation~\eqref{someEQ} becomes
\[
(3[d_1]_3+0) - (3[c_1]_3+0) = (3[x_1]_3+0) - (3[d_1]_3+0)
\]
and thus
\[
[d_1]_3 - [c_1]_3 = [x_1]_3 - [d_1]_3.
\]

We know we can find $[c_1]_3$ and $[d_1]_3$ to form an arithmetic sequence with $[x_1]_3$ through our induction hypothesis, so $[c_0]_3 = [\overline{c_1 0}]_3$ and $[d_0]_3 = [\overline{d_1 0}]_3$ form an arithmetic sequence with $[x_0]_3 = [\overline{x_1 0}]_3$.

\textbf{Case B:} If $x_{-1}$ = 2, set $c_{-1} = d_{-1} = 2$. Thus $c_0$ and $d_0$ both belong to upperZs, so $c_1$ and $d_1$ belong to the same row, and Equation~\eqref{someEQ} becomes
\[
(3[d_1]_3+2) - (3[c_1]_3+2) = (3[x_1]_3+2) - (3[d_1]_3+2)
\]
and thus
\[
[d_1]_3 - [c_1]_3 = [x_1]_3 - [d_1]_3.
\]
We know we can find $[c_1]_3$ and $[d_1]_3$ to form an arithmetic sequence with $[x_1]_3$ through our induction hypothesis, so $[c_0]_3 = [\overline{c_1 2}]_3$ and $[d_0]_3 = [\overline{d_1 2}]_3$ form an arithmetic sequence with $[x_0] = [\overline{x_1 2}]_3$.

\textbf{Case C:} If $x_{-1} = 1$, set $c_{-1} = 2$ and $d_{-1} = 0$. As $c_0$ is part of an upperZ and $d_0$ is part of a lowerZ, we have that $c_1$ is from row $2a$ and $d_1$ is from row $2a+1$. Equation~\eqref{someEQ} becomes
\[
(3[d_1]_3+0) - (3[c_1]_3+2) = (3[x_1]_3+1) - (3[d_1]_3+0)
\]
and thus
\begin{equation}
[d_1]_3-[c_1]_3  = [x_1]_3 + 1-[d_1]_3.
\label{secondEQ}
\end{equation}

To reiterate, we wish to determine that there exist two values $c_1$ and $d_1$ that satisfy Equation~\eqref{secondEQ} for each $x_1$. We do this by selecting specific values for the last digit of $c_1$ and $d_1$, which we can then use to gleam additional information from Equation~\eqref{someEQ} and Equation~\eqref{secondEQ}. 

We now consider cases within Case C, depending on the remainder of $2a$ modulo 3.

\textbf{Case 1:} $2a \equiv 0 \mod{3}$. The last digit of $c_1$ must be 0 or 1, and every term is part of an upperZ. The last digit of $d_1$ must be 2 or 0. If the last digit of $d_1$ is 2, then $d_1$ is part of an upperZ, so $d_2$ and $c_2$ will be in the same row. If the last digit of $d_1$ is 0, then $d_1$ is part of a lowerZ, and $c_2$ will be from the row below $d_2$. Now we consider cases depending on the last two digits of $x_0$.

If the last digit of $x_1$ is 0, we set the last digit of $c_1$ to $0$ and the last digit of $d_1$ to 2. Substituting this into Equation~\eqref{secondEQ}, we have
\[(3[d_2]_3+2)-(3[c_2]_3+0) = (3[x_2]_3+0)+1-(3[d_2]_3+2)\]
and thus
\begin{equation}\label{peculiar}
[d_2]_3-[c_2]_3 = ([x_2]_3 - 1) - [d_2]_3.
\end{equation}
with $c_2$ and $d_2$ in the same row. By Lemma \ref{lemma:minus1}, string $([x_2]_3-1)_3$ is from a row greater than or equal to that of $c_2$ and $d_2$. If string $([x_2]_3-1)_3$ is from the same row as $c_2$, we simply take $c_2 = d_2 = x_2$, otherwise we know by induction that we can always find $c_2$ and $d_2$ in satisfying Equation~\eqref{peculiar}. We call this a \textit{Peculiar} case because the induction process is simple, but $[c_2]_3$, $[d_2]_3$ form an arithmetic sequence with $[x_2]_3-1$, rather than $[x_2]_3$. Table~\ref{table:peculiar} shows the values used.
\begin{table}[ht!]
\begin{center}
\begin{tabular}{|c|c|c|c|c|c|c|}
\hline
 $2a \mod{3}$ & $c_{-2}$ & $d_{-2}$& $x_{-2}$ & equation & row of $d_2$ & case type\\
 \hline
 0 &  02 & 20  & 01 & $[d_2]_3-[c_2]_3 = [x_2]_3 - 1 - [d_2]_3$&  same as $c_2$ & Peculiar\\
  \hline
\end{tabular}
\caption{An example of the Peculiar case.}
\label{table:peculiar}
\end{center}
\end{table}

If the last digit of $x_1$ is 1, we set the last digit of $c_1$ to $1$ and the last digit of $d_1$ to $0$. Substituting this into Equation~\eqref{secondEQ}, we have
\[(3[d_2]_3+0)-(3[c_2]_3+1) = (3[x_2]_3+1)+1-(3[d_2]_3+0)\]
and thus
\[[d_2]_3-[c_2]_3 = [x_2]_3 +1 - [d_2]_3. \]
with $c_2$ in the row below $d_2$. The identical structure to Equation~\eqref{secondEQ} is instantly noticeable. By induction, as $c_2$ is from a row below $2a$, we know we can find $c_2$ and $d_2$ in this case. We call this an \textit{Iterative} case because the explicit process of finding $c_2$ and $d_2$ requires repeated application of this caseworks. Table~\ref{table:iterative} shows the values used.
\begin{table}[ht!]
\begin{center}
\begin{tabular}{|c|c|c|c|c|c|c|}
\hline
 $2a \mod{3}$  & $c_{-2}$ & $d_{-2}$& $x_{-2}$ & equation & row of $d_2$ & case type \\
 \hline
 0 &  12 &00  & 11 & $[d_2]_3-[c_2]_3 = [x_2]_3+1 - [d_2]_3$&  above $c_2$ & Iterative\\
  \hline
\end{tabular}
\caption{An example of the Iterative case.}
\label{table:iterative}
\end{center}
\end{table}

If the last digit of $x_1$ is 2, we set the last digit of $c_1$ to $1$ and the last digit of $d_1$ to $2$. Substituting this into Equation~\eqref{secondEQ}, we have
\[(3[d_2]_3+2)-(3[c_2]_3+1) = (3[x_2]_3+2)+1-(3[d_2]_3+2)\]
and thus
\[[d_2]_3-[c_2]_3 = [x_2]_3 - [d_2]_3.\]
with $c_2$ and $d_2$ in the same row. The identical structure to Equation~\eqref{someEQ} is instantly noticeable. By induction, as $c_2$ is from a row below $2a$, we know we can find $c_2$ and $d_2$ in this case. We call this the \textit{Simplest} case simply as $[c_2]_3$, $[d_2]_3$ and $[x_2]_3$ form an arithmetic sequence.  Table~\ref{table:simplest} shows the values used.

\begin{table}[ht!]
\begin{center}
\begin{tabular}{|c|c|c|c|c|c|c|}
\hline
 $2a \mod{3}$  & $c_{-2}$ & $d_{-2}$& $x_{-2}$ & equation & row of $d_2$ & case type \\
 \hline
 0 &  12 &20  & 21 & $[d_2]_3-[c_2]_3 = [x_2]_3 - [d_2]_3$&  same as $c_2$ & Simplest\\
  \hline
\end{tabular}
\caption{An example of the Simplest case.}
\label{table:simplest}
\end{center}
\end{table}

\textbf{Cases 2 and 3:} $2a \equiv 1, 2 \mod{3}$. The induction process for these cases proceeds in a similar fashion to either the Peculiar, Iterative, or Simplest case. All results are in Table~\ref{table:allcases}.

\begin{table}[ht!]
\begin{center}
\begin{tabular}{|c|c|c|c|c|c|c|}
\hline
 $2a \mod{3}$  & $c_{-2}$ & $d_{-2}$& $x_{-2}$ & equation & row of $d_2$ & case type \\
 \hline
 0 &  02 &20  & 01 & $[d_2]_3-[c_2]_3 = [x_2]_3 - 1 - [d_2]_3$&  same as $c_2$ & Peculiar \\
 \hline
 0 &  12 &00  & 11 & $[d_2]_3-[c_2]_3 = [x_2]_3+1 - [d_2]_3$&  above $c_2$ & Iterative \\
 \hline
  0 &  12 &20  & 21 & $[d_2]_3-[c_2]_3 = [x_2]_3 - [d_2]_3$&  same as $c_2$ & Simplest \\
 \hline
  1 &  02 &20  & 01 & $[d_2]_3-[c_2]_3 = [x_2]_3 - 1 - [d_2]_3$&  same as $c_2$ & Peculiar \\
 \hline
  1 &  02 &10  & 11 & $[d_2]_3-[c_2]_3 = [x_2]_3 - [d_2]_3$&  same as $c_2$ & Simplest \\
 \hline
  1 &  22 & 10  & 21 & $[d_2]_3-[c_2]_3 = [x_2]_3+1 - [d_2]_3$&  above $c_2$ & Iterative \\
 \hline
  2 &  22 &00  & 01 & $[d_2]_3-[c_2]_3 = [x_2]_3+1 - [d_2]_3$&  above $c_2$ & Iterative \\
 \hline
   2 &  12 &00  & 11 & $[d_2]_3-[c_2]_3 = [x_2]_3+1 - [d_2]_3$&  above $c_2$ & Iterative \\
 \hline
   2 &  22 &10  & 21 & $[d_2]_3-[c_2]_3 = [x_2]_3+1 - [d_2]_3$&  above $c_2$ & Iterative \\
 \hline
\end{tabular}
\caption{All cases.}
\label{table:allcases}
\end{center}
\end{table}

We have found for any possible $x_0$ a process where $c_0$ and $d_0$ must exist. Thus, two elements from row $3a+1$ evaluated in base 3 form an arithmetic sequence with any $[x_0]_3$ from a later row, so row $3a+1$ evaluated in base 3 is the same set of numbers as $R_{3a+1}$.
\end{proof}

\subsection{Proof of Theorem~\ref{thm:interm}}

\begin{proof}[Proof of Theorem~\ref{thm:interm}]
By combining the previous three subsections, Lemma~\ref{lemma:3a}, Lemma~\ref{lemma:3a+2}, and Lemma~\ref{lemma:3a+1}, we see that if for $i<j$ every row $i$ evaluated in base 3 is the same set of numbers as $R_i$, then row $j$ evaluated in base 3 is the same set of numbers as $R_j$. This is precisely the induction hypothesis we have set out to prove; thus, every row $i$ evaluated in base 3 is the same set of numbers as $R_i$.
\end{proof}

\section{The proof of Theorems~\ref{thm:main} and \ref{thm:Scs}}\label{sec:twoproofs}

Remember that $S_i$ is the $i$-th 3-free sequence of the greedy partition of non-negative integers. Also, recall the definition of $R_i$ as it appears below Theorem~\ref{thm:interm}.

\begin{lemma} $R_i = S_i$.\label{lemma:RiSi}
\end{lemma}

\begin{proof}
Suppose $a$ is the smallest number that belongs to $S_i$ and $R_j$, where $i \neq j$, that is, the smallest number that is out of place. 

If $i < j$, then $a$ can not be the last term of a 3-term sequence with terms in $S_i$, which contradicts a required property to belong to $R_j$. If $i > j$, then $a$ is the last term of an arithmetic progression with elements in $S_j$, which contradicts $a$ belonging to $R_j$.
\end{proof}

Now we are ready to prove Theorem~\ref{thm:main} and Theorem~\ref{thm:Scs}.

\begin{proof}[Proof of Theorem~\ref{thm:main}]
By Lemma~\ref{lemma:RiSi}, Theorem~\ref{thm:main} is equivalent to Theorem~\ref{thm:interm}.
\end{proof}

\begin{proof}[Proof of Theorem~\ref{thm:Scs}.]
By construction the first term of row $i$ is $2i$ written in base $\frac{3}{2}$.

On the other hand, by Theorem~\ref{thm:main} row $i$ contains the same set of numbers as integers in sequence $S_i$ written in base 3. By Lemma~\ref{lemma:zerocolumn}, the first term of each sequence written in base 3 is in the zeroth column. Thus the zeroth column represents the Stanley cross-sequence written in base 3.
\end{proof}

\section{Acknowledgements}

We are grateful to the MIT PRIMES program for giving us the opportunity to undergo this research.

\end{document}